\documentclass{amsart}
\usepackage{amsthm, amssymb, amsfonts, amscd}
\usepackage{graphicx}
\usepackage{xcolor}
\usepackage[colorlinks = true,
linkcolor = black,
urlcolor = blue,
citecolor = black]{hyperref}
\usepackage[margin=1.0in]{geometry}
\usepackage[T1]{fontenc}
\allowdisplaybreaks
\DeclareMathOperator{\Prep}{Prep}
\DeclareMathOperator{\prim}{prim}
\newcommand{\bb}{\mathbb}
\newcommand{\mf}{\mathbf}
\newcommand{\mc}{\mathcal}
\newcommand{\bZ}{\bb{Z}}
\newcommand{\bQ}{\bb{Q}}
\newcommand{\bR}{\bb{R}}
\DeclareMathOperator{\vol}{vol}
\newcommand{\cB}{\mc{B}}

\newcommand{\cP}{\mc{P}}
\newcommand{\ve}{\varepsilon}
\newcommand{\vt}{\vartheta}

\theoremstyle{definition}
\newtheorem{theorem}{Theorem}[section]

\newtheorem{proposition}[theorem]{Proposition}
\newtheorem{lemma}[theorem]{Lemma}

\begin{document}

\title[Average number of rational preperiodic points]{The average number of rational preperiodic points of polynomials over $\bQ$}
\author{Jungin Lee}
\date{}
\address{J. Lee -- Department of Mathematics, Ajou University, Suwon 16499, Republic of Korea}
\email{jileemath@ajou.ac.kr}

\begin{abstract}
Let $M_d(X)$ denote the average number of rational preperiodic points among degree-$d$ polynomials over $\mathbb{Q}$ with vanishing $z^{d-1}$ coefficient, constant term $1$, and height at most $X$. We prove that, for every integer $d\ge2$, $M_d(X) \sim \gamma_d X^{-1}$ for an explicit constant $\gamma_d>0$. For $d\ge4$, the error term is $O_{d,\varepsilon}(X^{-2+\varepsilon})$ for every $\varepsilon>0$; for $d=3$ the error term is $O(X^{-3/2})$; and for $d=2$ the error term is $O(X^{-3/2}\log X)$.
\end{abstract}
\maketitle

%------------------------------------------------
%------------------------------------------------
\vspace{-7mm}
\section{Introduction}

The dynamical uniform boundedness conjecture of Morton and Silverman \cite{MS94} predicts that, for each integer $d\ge2$, the number of rational preperiodic points of a degree-$d$ polynomial over $\bQ$ is bounded by a constant depending only on $d$. This conjecture remains open even for $d=2$.
Following the statistical approach of Le Boudec and Mavraki \cite{BM21}, we instead order polynomials by height and study the number of rational preperiodic points on average.
This perspective allows us to measure how frequently rational preperiodic points occur among polynomials of bounded height and, more precisely, how their average number varies as the height bound tends to infinity.

For nonnegative functions $f(X)$ and $g(X)$, we write $f(X) \ll_{\alpha_1, \ldots, \alpha_k} g(X)$ (or $g(X) \gg_{\alpha_1, \ldots, \alpha_k} f(X)$) if $f(X)\le Cg(X)$ for all sufficiently large $X$, where the constant $C>0$ depends only on $\alpha_1, \ldots, \alpha_k$. 
For a positive integer $n$, let $\tau(n)$ be the number of positive divisors of $n$. 

For an integer $d\ge2$, let 
\begin{equation*}
\bZ_{\prim}^d := \{ (x_1,\ldots,x_d)\in\bZ^d : \gcd(x_1,\ldots,x_d)=1 \}
\end{equation*}
be the set of primitive vectors in $\bZ^d$. 
For every $\mf{a}=(a_d,a_{d-2},\ldots,a_0) \in \bZ^d$ with $a_0\ne0$, define
\begin{equation*}
\psi_{\mf{a}}(z) := \frac{a_d}{a_0}z^d+\sum_{i=1}^{d-2} \frac{a_i}{a_0}z^i+1 \in \bQ[z],
\end{equation*}
and 
\begin{equation*}
\cP_d := \left\{ \psi_{\mf{a}} \in\bQ[z] :
\mf{a}=(a_d,a_{d-2},\ldots,a_0)\in\bZ_{\prim}^d, \, a_d\ne0, \, a_0>0 \right\}.
\end{equation*}
Each $\psi\in\cP_d$ has a unique primitive representative $\mf a$
with $a_0>0$. We define its height by
\begin{equation*}
H(\psi_{\mf{a}}) := \max(|a_d|,|a_{d-2}|,\ldots,|a_1|,a_0).
\end{equation*}
For a real number $X \ge1$, set
\begin{equation*}
\cP_d(X)=\{\psi\in\cP_d:H(\psi)\le X\}.
\end{equation*}

For a self-map $f$ on a set $S$ and an integer $n\ge0$, write $f^n$ for the $n$-th iterate of $f$, with $f^0$ the identity map. For $f\in\bQ[z]$, let $\Prep_{\bQ}(f)$ denote the set of rational preperiodic points of $f$; that is,
\begin{equation*}
\Prep_{\bQ}(f)=\{z\in\bQ: f^m(z)=f^n(z) \text{ for some integers } m>n\ge0 \}.
\end{equation*}
Set
\begin{equation*}
A_d(X)=\sum_{\psi\in\cP_d(X)}\#\Prep_{\bQ}(\psi),
\quad M_d(X)=\frac{A_d(X)}{\#\cP_d(X)}.
\end{equation*}

Although the dynamical uniform boundedness conjecture remains open, Le Boudec and Mavraki \cite{BM21} obtained a quantitative upper bound for the average number of rational preperiodic points among polynomials in $\cP_d(X)$.
\begin{theorem} \label{thm1a}
(\cite[Theorem 1.1]{BM21}) Let $\vt_2 = 1/2$ and $\vt_d=\frac{2(d+1)}{5d+1}$ for each $d \ge 3$. Then for every $\ve>0$,
\begin{equation*}
M_d(X)\ll_{d,\ve}X^{-\vt_d+\ve}.
\end{equation*}
\end{theorem}
We note that their proof \cite[p. 18]{BM21} actually gives a better upper bound $M_2(X)\ll X^{-1/2}\log X$.

For $c=a/b\in\bQ$ written in lowest terms with $b>0$, let $H(c):=\max(|a|,b)$. For the quadratic family $f_c(z)=z^2+c$, the remainder term $\mc{R}(X)$ in the work of Olechnowicz \cite[Theorem 8.1]{Ole26} satisfies
\begin{equation*}
\mc{R}(X)=\sum_{\substack{c\in\bQ\\H(c)\le X}}\#\Prep_{\bQ}(f_c).
\end{equation*}
Indeed, Olechnowicz counts preperiodic points on $\mf{P}^1(\bQ)$ and subtracts the contribution of the fixed point $\infty$. By the bijection explained at the beginning of Section \ref{Sec3}, the contribution from $c\ne0$ is exactly $A_2(X)$. Since $\Prep_{\bQ}(f_0)=\{-1,0,1\}$, we have $\mc{R}(X)=A_2(X)+3$. Thus \cite[Theorem 8.1]{Ole26} gives
\begin{equation} \label{eq1a}
X \ll A_2(X)\ll X^{3/2}\log\log X\log\log\log X.
\end{equation}
Moreover, \cite[(2.20)]{BM21} gives
\begin{equation} \label{eq1b}
\#\cP_d(X)=\frac{2^{d-1}}{\zeta(d)}X^d+O_d(X^{d-1}\log X)
\end{equation}
for every $d \ge 2$, where $\zeta$ denotes the Riemann zeta function. By \eqref{eq1a} and \eqref{eq1b}, we have 
\begin{equation*}
M_2(X)\ll X^{-1/2}\log\log X\log\log\log X.
\end{equation*}

Le Boudec and Mavraki conjectured that for every integer $d\ge 2$, there exists a constant $\gamma_d>0$ such that $M_d(X)\sim \gamma_dX^{-1}$; see \cite[(1.3)]{BM21}. The purpose of this paper is to prove this conjecture, with a power-saving error term.

\begin{theorem} \label{thm1b}
For every integer $d\ge2$, there is a constant $\gamma_d>0$ such that
\begin{equation*}
M_d(X)=\frac{\gamma_d}{X}+\begin{cases}
O(X^{-3/2}\log X),&d=2,\\
O(X^{-3/2}),&d=3,\\
O_{d,\ve}(X^{-2+\ve}),&d\ge4.
\end{cases}
\end{equation*}
In the case $d\ge4$, the estimate holds for every $\ve>0$.
For $d\ge3$, the constant $\gamma_d$ is given by the absolutely convergent
series in \eqref{eq1c} below, while
\begin{equation*}
\gamma_2=\frac{1+\sqrt{5}}2
+2\log \frac{1+\sqrt{5}}2
+\frac{2\pi}{3\sqrt{3}}.
\end{equation*}
\end{theorem}

For the quadratic family $z^2+c$, Olechnowicz \cite[Theorem 8.1]{Ole26} obtained a lower bound for $A_2(X)$ whose main term, after division by $\#\cP_2(X)$, gives $\gamma_2X^{-1}$. He further obtained the corresponding asymptotic formula under the assumption that no quadratic polynomial over $\bQ$ has a rational cycle of length greater than $3$. The proof of Theorem \ref{thm1b} in the case $d=2$ does not require such an assumption.

We now describe the constant $\gamma_d$ for $d \ge 3$. Set 
\begin{equation*}
\cB_d=[-1,1]^{d-1}\times[0,1]\subset\bR^d.
\end{equation*}
For a primitive vector $\ell\in\bZ^d\setminus\{0\}$, define
\begin{equation*}
V_d(\ell)=\frac{\vol_{d-1}(\cB_d\cap\ell^\perp)}{\|\ell\|_2},
\end{equation*}
where $\vol_{d-1}$ denotes Euclidean volume on the hyperplane $\ell^\perp$ and $\|\ell\|_2$ is the Euclidean norm. For integers $x$ and $y$ with $x \ne 0$, $y \ge 1$ and $\gcd(x,y)=1$, set
\begin{equation*}
\ell_d(x,y)=(x^d,x^{d-2}y^2,\ldots,xy^{d-1},y^{d-1}(y-x)) \in \bZ^d.
\end{equation*}
Since $\gcd(x^d,y^{d-1}(y-x))=1$, the vector $\ell_d(x,y)$ is primitive. Moreover, $\psi_{\mf{a}}(x/y)=x/y$ if and only if $\ell_d(x,y)\cdot\mf{a}=0$.
Also set
\begin{equation*}
\ell_d^{(1)}=(1,\ldots,1,0)\in \bZ^d,\quad
\ell_d^{(0)}=(1,\ldots,1,1)\in \bZ^d,\quad
\eta_d=\begin{cases}2,&d=3,\\1,&d\ge4.\end{cases}
\end{equation*}
We use the notation $\ell_d^{(1)}$ and $\ell_d^{(0)}$ since for $\mf a=(a_d,a_{d-2},\ldots,a_1,a_0) \in \bZ_{\prim}^d$ with $a_0 \ne 0$, $\psi_{\mf a}(1)=1$ (resp. $\psi_{\mf a}(1)=0$) if and only if $\ell_d^{(1)}\cdot\mf a=0$ (resp. $\ell_d^{(0)}\cdot\mf a=0$). Then
\begin{equation} \label{eq1c}
\gamma_d=\frac{\zeta(d)}{2^{d-1}\zeta(d-1)}
\left( \sum_{\substack{x\ne0,\ y\ge1\\ \gcd(x,y)=1}} V_d(\ell_d(x,y))
+\eta_dV_d(\ell_d^{(1)})+2V_d(\ell_d^{(0)})
\right).
\end{equation}

The sum over $(x,y)$ in \eqref{eq1c} gives the contribution of rational fixed points to $\gamma_d$. When $\psi(1)=1$, the summand indexed by $(x,y)=(1,1)$ accounts for the fixed point $1$. The term $\eta_dV_d(\ell_d^{(1)})$ accounts for its preimage $0$ and, when $d=3$, its additional preimage $-1$. In the discussion preceding \cite[(1.3)]{BM21}, Le Boudec and Mavraki expected the contribution from polynomials with a rational periodic point of length at least $2$ to be negligible compared with that from polynomials having a rational fixed point. However, when $\psi(1)=0$, the two points in the cycle $0\mapsto1\mapsto0$ account for the term $2V_d(\ell_d^{(0)})$ in \eqref{eq1c}.

The paper is organized as follows. We prove Theorem \ref{thm1b} for the cases $d\ge3$ and $d=2$ in Sections \ref{Sec2} and \ref{Sec3}, respectively. For $d\ge3$, fix three consecutive orbit points $z$, $z_1=f(z)$, $z_2=f^2(z)$ for $f=\psi_{\mf a}$ and all entries of $\mf a$ except $a_0$ and $a_1$. The relations $f(z)=z_1$ and $f(z_1)=z_2$ then give two linear equations in $a_0$ and $a_1$. Proposition \ref{prop2b} uses a divisor estimate to bound the contribution to $A_d(X)$ from rational preperiodic points $z$ for which these two equations are linearly independent, i.e., for which $\Delta_f(z)\ne0$ in the notation of Section 2.1. Lemmas \ref{lem2c} and \ref{lem2d} then reduce the problem to counting rational fixed points and the contributions from the cases $f(1)=1$ and $f(1)=0$. We first count the relevant integral coefficient vectors without imposing primitivity and apply Möbius inversion after summing over the fixed points.

For $d=2$, we conjugate $\psi_{\mf a}$ to $f_c(z)=z^2+c$ and use the recurrence for seven consecutive orbit points $z,f_c(z),\ldots,f_c^6(z)$ to define a projective surface $\mc S$ of degree $32$. The map $(c,z)\mapsto[1:z:f_c(z):f_c^6(z)]$ sends the pairs counted by $A_2(X)$ injectively to rational points on $\mc S$ of height at most $2\sqrt X$. We classify the geometric lines and geometrically integral conics on $\mc S$, and apply Salberger's estimates to show that the pairs $(c,z)$ whose images lie on none of the four affine lines contribute $O(\sqrt X)$ to $A_2(X)$. The four affine lines correspond to the cases in which $z$ or $-z$ is fixed or has period $2$. Counting the points on these lines gives the constant $\gamma_2$ obtained by Olechnowicz \cite{Ole26}.

%------------------------------------------------
%------------------------------------------------
\section{Proof of Theorem \ref{thm1b}: the case \texorpdfstring{$d\ge3$}{d>=3}} \label{Sec2}

In this section, fix an integer $d\ge3$. Whenever $z\in\bQ$ is written as $z=x/y$, we assume that $x,y\in\bZ$, $y>0$, and $\gcd(x,y)=1$. For each prime $p$, let $v_p(n)$ denote the exponent of $p$ in a nonzero integer $n$, and extend $v_p$ to $\bQ$ in the usual way with $v_p(0)=\infty$. Set $|0|_p=0$ and $|u|_p=p^{-v_p(u)}$ for $u\in\bQ^\times$.

%------------------------------------------------
\subsection{Arithmetic constraints on rational preperiodic points}

Following Le Boudec and Mavraki \cite[(2.4)]{BM21}, for $\psi_{\mf a}\in\cP_d$ and each place $v$ of $\bQ$, set
$$
r_v(\psi_{\mf a}):=\max\left\{\left|\frac{a_0}{a_d}\right|_v^{1/(d-1)},\ \max_{0\le i\le d-2}\left|\frac{a_i}{a_d}\right|_v^{1/(d-i)}\right\}.
$$

\begin{lemma} \label{lem2a}
Let $\psi_{\mf{a}}\in\cP_d(X)$ for a real number $X \ge 1$ and $z=x/y\in\Prep_{\bQ}(\psi_{\mf{a}})\cap\bQ^\times$. Then $y^2\mid a_d$ and
\begin{equation*}
|x| \le dy\left(\frac{X}{|a_d|}\right)^{1/2}.
\end{equation*}
\end{lemma}

\begin{proof}
Since $z$ is preperiodic, its local canonical height vanishes at every place. Applying \cite[Lemmas 2.2 and 2.3]{BM21}, we obtain
$$
|z|_p\le r_p(\psi_{\mf a})\quad\text{and}\quad |z|\le d\,r_\infty(\psi_{\mf a}).
$$
Since $d\ge3$ and $i\le d-2$, every exponent in the definition of $r_v(\psi_{\mf a})$ is at most $1/2$. At a finite place $p$, each ratio $|a_i/a_d|_p$ is at most $|a_d|_p^{-1}$ since $a_i\in\bZ$. At the archimedean place, each ratio $|a_i/a_d|$ is at most $H(\psi_{\mf a})/|a_d|$. Both upper bounds are at least $1$, so
$$
|z|_p\le |a_d|_p^{-1/2}\quad\text{and}\quad |z|\le d\left(\frac{H(\psi_{\mf a})}{|a_d|}\right)^{1/2}.
$$
For each prime $p\mid y$, the first inequality gives $2v_p(y)\le v_p(a_d)$, so $y^2\mid a_d$. The second gives the desired bound for $|x|$ since $H(\psi_{\mf a})\le X$.
\end{proof}

For $\psi_{\mf{a}}\in\cP_d(X)$, write $a_d=\sigma m q^2$, where $\sigma \in \{-1,1\}$, $m \ge 1$ is squarefree and $q \ge 1$. By Lemma \ref{lem2a}, every point in $\Prep_{\bQ}(\psi_{\mf{a}})$ has the form $u/q$, where
\begin{equation} \label{eq2a}
|u|\le d(X/m)^{1/2},\quad q\le(X/m)^{1/2}.
\end{equation}
We note that the fraction $u/q$ need not be reduced. For $f\in\cP_d$, define
\begin{align*}
\Delta_f(z)&=z(f^2(z)-1)-f(z)(f(z)-1)\quad(z\in\bQ),\\
D_f&=\#\{z\in\Prep_{\bQ}(f):\Delta_f(z)\ne0\}.
\end{align*}
Once the three orbit points $z,f(z),f^2(z)$ and all coefficients other than $a_0$ and $a_1$ are fixed, the two orbit relations are linear in $a_0$ and $a_1$. Their coefficient matrix has rows $(1-f(z),z)$ and $(1-f^2(z),f(z))$, so its determinant is $\Delta_f(z)$. Thus, when $\Delta_f(z)\ne0$, the two relations uniquely determine $a_0$ and $a_1$.

\begin{proposition} \label{prop2b}
For every $\ve>0$,
\begin{equation} \label{eq2b}
\sum_{f\in\cP_d(X)}D_f\ll_{d,\ve}X^{d-2+\ve}.
\end{equation}
\end{proposition}

\begin{proof}
Let $f=\psi_{\mf{a}}$, write $a_d=\sigma m q^2$ as above and set $B=(X/m)^{1/2}$. For a point counted by $D_f$, write
\begin{equation*}
z=u/q,\quad f(z)=v/q,\quad f^2(z)=w/q.
\end{equation*}
By Lemma \ref{lem2a}, $1\le q\le B$ and $|u|,|v|,|w|\le dB$. The integer
\begin{equation*}
D:=u(w-q)-v(v-q)=q^2\Delta_f(z)
\end{equation*}
is nonzero. The relations $f(u/q)=v/q$ and $f(v/q)=w/q$ give
\begin{align*}
q^{d-3}\bigl(a_1u-a_0(v-q)\bigr)
&=-\sigma mu^d-\sum_{i=2}^{d-2}a_iq^{d-i-2}u^i,\\
q^{d-3}\bigl(a_1v-a_0(w-q)\bigr)
&=-\sigma mv^d-\sum_{i=2}^{d-2}a_iq^{d-i-2}v^i.
\end{align*}
Eliminating $a_1$, we obtain
\begin{equation*}
a_0q^{d-3}D=N:=\sigma muv(v^{d-1}-u^{d-1})+\sum_{i=2}^{d-2}a_iq^{d-i-2}uv(v^{i-1}-u^{i-1}).
\end{equation*}
Since $D \ne 0$, we have $uv \ne 0$ and $0<|N|\ll_d X^{(d+1)/2}$.
For fixed $\sigma,m,q,u,v,a_2,\ldots,a_{d-2}$, a signed divisor $D$ of $N$ determines at most one $w$ and then at most one pair $(a_0,a_1)$. The number of signed divisors of $N$ is $O_{d,\ve}(X^\ve)$.

% The case d=3
Suppose first that $d=3$. Then $N=\sigma muv(v^2-u^2)$ is independent of $q$. Fix $\sigma,m,u,v$ and a signed divisor $D$ of $N$. Since
\begin{equation*}
qv\equiv D+v^2\pmod{|u|},
\end{equation*}
the number of possible values of $q$ in $[1,B]$ is
\begin{equation*}
O\left(1+\frac{B\gcd(u,v)}{|u|}\right).
\end{equation*}
Moreover,
\begin{equation*}
\sum_{0<|v|\le 3B}\gcd(u,v)=2\sum_{1\le v\le 3B}\gcd(u,v)\le 2\sum_{\delta\mid u}\delta\#\{1\le v\le 3B:\delta\mid v\}\ll B\tau(|u|).
\end{equation*}
Hence
\begin{equation*}
\sum_{0<|u|,|v|\le 3B}\left(1+\frac{B\gcd(u,v)}{|u|}\right)\ll B^2+B^2\sum_{1\le |u|\le 3B}\frac{\tau(|u|)}{|u|}\ll B^2(\log(2B))^2,
\end{equation*}
where the last inequality follows from
\begin{equation*}
\sum_{n\le 3B}\frac{\tau(n)}n=\sum_{ab\le 3B}\frac{1}{ab}\le\left(\sum_{a\le 3B}\frac{1}{a}\right)^2\ll (\log(2B))^2.
\end{equation*}

Since $0<|N|\ll X^2$, for each fixed $\sigma,m,u,v$, there are $O_\ve(X^{\ve/2})$ possible signed divisors $D$ of $N$. Hence, for fixed $\sigma$ and $m$, the number of possible quadruples $(u,v,D,q)$ is $O_\ve(X^{\ve/2}B^2(\log(2B))^2)$. Since these choices determine at most one triple $(w,a_0,a_1)$, summing over $\sigma\in\{\pm1\}$ and $m\le X$ gives
\begin{equation*}
\sum_{f\in\cP_3(X)}D_f\ll_\ve X^{\ve/2}\sum_{m\le X}B^2(\log(2B))^2\ll_\ve X^{1+\ve/2}(\log X)^3\ll_\ve X^{1+\ve}.
\end{equation*}

% The case d\ge4
Now suppose that $d\ge4$. For fixed
$\sigma,m,u,v,a_2,\ldots,a_{d-2}$, regard $N$ as a polynomial in $q$:
\begin{equation*}
N(q)=c_0+c_1q+\cdots+c_{d-4}q^{d-4}.
\end{equation*}
Since $N \ne 0$, we have $c_j \ne 0$ for some $j$. Let $j_0 \le d-4$ be the lowest value of $j$ satisfying $c_j \ne 0$. Then $q^{d-3}\mid N(q)=q^{j_0}(c_{j_0}+qR(q))$ implies that $q\mid c_{j_0}$. Each coefficient satisfies $|c_j|\ll_d X^{(d+1)/2}$. Hence there are $O_{d,\ve}(X^\ve)$ choices for $q$, and then $O_{d,\ve}(X^\ve)$ choices for $D$.
There are $O_d(X^{d-3})$ choices for $a_2,\ldots,a_{d-2}$ and $O_d(B^2)$ choices for $(u,v)$. Summing over $\sigma\in\{\pm1\}$ and $m\le X$ gives
\begin{equation*}
\sum_{f\in\cP_d(X)}D_f
\ll_{d,\ve} \sum_{m\le X} X^\ve X^\ve X^{d-3} B^2
\ll_{d,\ve} X^{d-3+2\ve}\sum_{m\le X}\frac{X}{m}
\ll_{d,\ve} X^{d-2+2\ve}\log X.
\end{equation*}
Replacing $\ve$ with $\ve/3$ proves \eqref{eq2b}.
\end{proof}

%------------------------------------------------
\subsection{Reduction to fixed points and special orbits}

For $z\ne0$, the slope $\lambda=(f(z)-1)/z$ determines the affine map $w\mapsto\lambda w+1$ which sends $z$ to $f(z)$. If $zf(z)\ne0$, then $\Delta_f(z)=0$ if and only if the same map sends $f(z)$ to $f^2(z)$. Such an affine map over $\bQ$ can have rational periodic points of order at least $2$ only when $\lambda=-1$. When $\lambda=-1$, the affine map is $w\mapsto1-w$, whose $2$-cycles are precisely the pairs $\{z,1-z\}$ with $z\ne1/2$. This explains the appearance of the pairs $z,1-z$ below.

For $f\in\cP_d$, let
\begin{align*}
F_f&=\#\{z\in\bQ:f(z)=z\}, \\
T_f&=\#\{z\in\bQ\setminus\{0,1\}:f(z)=1-z,\ f(1-z)=z,\ z\ne1/2\}.
\end{align*}
For $j\in\{0,1\}$, write $I_j(f)=1$ if $f(1)=j$, and $I_j(f)=0$ otherwise. Then $T_f+2I_0(f)$ counts the points in $2$-cycles whose elements sum to $1$, and both $F_f$ and $T_f+2I_0(f)$ are at most $d$ since the points counted by $F_f$ (resp. $T_f+2I_0(f)$) are distinct roots of the degree-$d$ polynomial $f(z)-z$ (resp. $f(z)+z-1$).
Set
\begin{align*}
R_f&=\begin{cases}
\#\{z\in\bQ:f(z)=1\}-1-\eta_d,&f(1)=1,\\
0,&f(1)\ne1,
\end{cases}\\
U_f&=\begin{cases}
\#\{z\in\bQ\setminus\{1\}:f(z)=0\},&f(1)\in\{0,1\},\\
0,&f(1)\notin\{0,1\}.
\end{cases}
\end{align*}
The quantity $R_f$ is nonnegative: if $f(1)=1$, then $0$ and $1$ are distinct preimages of $1$; when $d=3$, the same condition gives
\begin{equation} \label{eq2d}
f(z)=\frac{a_3}{a_0}(z^3-z)+1,
\end{equation}
so the preimages of $1$ are exactly $0$, $1$ and $-1$. For every $d\ge3$, if $f(1)=0$, then every nonzero rational preimage $z$ of $1$ is already counted by $D_f$, since $\Delta_f(z)=-z\ne0$.

\begin{lemma} \label{lem2c}
For every $f\in\cP_d$,
\begin{equation} \label{eq2e}
0\le\#\Prep_{\bQ}(f)-F_f-\eta_d I_1(f)-2I_0(f)
\ll_d D_f+R_f+T_f+U_f.
\end{equation}
\end{lemma}

\begin{proof}
For the lower bound, if $f(1)=1$, then $0$ is a nonfixed preperiodic point; when $d=3$, \eqref{eq2d} gives the additional nonfixed preperiodic point $-1$. If $f(1)=0$, then $0$ and $1$ form a $2$-cycle. Since these two conditions are mutually exclusive, we have
\begin{equation*}
\#\Prep_{\bQ}(f)\ge F_f+\eta_d I_1(f)+2I_0(f).
\end{equation*}

Now we prove the upper bound. For $z\in \Prep_{\bQ}(f)\setminus\{0\}$, set $t_f(z)=(f(z)-1)/z$. Each fiber of $t_f$ contains at most $d-1$ points. If $zf(z)\ne0$, then
\begin{equation} \label{eq2f}
\Delta_f(z)=zf(z)\left(t_f(f(z))-t_f(z)\right).
\end{equation}
The only finite cycles of the affine map $L_\lambda(z)=\lambda z+1$ ($\lambda\in\bQ$) are fixed points, except when $\lambda=-1$, in which case $2$-cycles also occur. When $\lambda=1$, we have $L_1^m(z)=z+m$ for $m\ge1$, so $L_1$ has no rational periodic points. For $\lambda\ne1$, the assertion follows from
\begin{equation*}
L_\lambda^m(z)-z=\frac{\lambda^m-1}{\lambda-1}(L_\lambda(z)-z) \quad (m \ge 1).
\end{equation*}
Moreover, when $\lambda\ne0$, the map $L_\lambda$ is injective so every preperiodic point is periodic.

For each nonempty fiber $t_f^{-1}(\lambda)$, either $f(z)$ does not lie in the same fiber for some $z\in t_f^{-1}(\lambda)$, or the fiber contains a cycle of $f$. In the first case, if $f(z)\ne0$, then $z$ is counted by $D_f$ by \eqref{eq2f}; if $f(z)=0$, then $z$ is one of at most $d$ preimages of $0$. In the second case, the cycle is also a cycle of $L_\lambda$ and hence contains a point counted by $F_f+(T_f+2I_0(f))$. Consequently,
\begin{equation*}
\#\Prep_{\bQ}(f)\le(d-1)(D_f+d+F_f+T_f+2I_0(f))+1
\le(d-1)(D_f+3d)+1.
\end{equation*}
If $D_f+R_f+T_f+U_f>0$, then the upper bound of \eqref{eq2e} holds.

Now we assume that $D_f=R_f=T_f=U_f=0$. Suppose first that $0\notin\Prep_{\bQ}(f)$. Then $zf(z) \ne 0$ for every $z \in \Prep_{\bQ}(f)$, so $D_f=0$ and \eqref{eq2f} imply $t_f(f(z))=t_f(z)$. Thus $t_f$ is constant along every preperiodic $f$-orbit. This constant cannot be $0$, since then the $f$-orbit would contain $1$, and $f(0)=1$ would imply that $0$ is preperiodic. Therefore each preperiodic $f$-orbit is a finite orbit of $L_\lambda$ for some $\lambda\ne0$. Since $L_\lambda$ is injective and $T_f=0$, every preperiodic point is fixed. Hence $\#\Prep_{\bQ}(f)=F_f$ and $I_1(f)=I_0(f)=0$.

Suppose now that $0\in\Prep_{\bQ}(f)$ (so $1=f(0)\in\Prep_{\bQ}(f)$) and write $\lambda_1=t_f(1)=f(1)-1$. Until the $f$-orbit of $1$ reaches $0$, it agrees with the orbit under $L_{\lambda_1}$. If the $f$-orbit of $1$ does not contain $0$, the finiteness of the orbit gives $\lambda_1=0$. Indeed, if $\lambda_1\ne0$, then $L_{\lambda_1}$ is injective, so $1$ is periodic under $L_{\lambda_1}$. Since $L_{\lambda_1}(1)=1+\lambda_1\ne1$, $1$ is not a fixed point so $\lambda_1=-1$. Then $f(1)=\lambda_1+1=0$, a contradiction. If $f^n(1)=0$ and $f^m(1) \ne 0$ for all $0 \le m \le n-1$, then $f^n(1)=1+\lambda_1+\cdots+\lambda_1^n=0$ for $\lambda_1 \in \bQ$ so $\lambda_1=-1$. Hence $f(1)\in \{0,1\}$.

If $f(1)=1$, then $U_f=0$ implies that $0$ has no rational preimage. Hence for every $z \in \Prep_{\bQ}(f) \setminus \{0\}$, the $f$-orbit of $z$ does not contain $0$ so $D_f=0$ and \eqref{eq2f} imply that $t_f$ is constant along the orbit.
If this constant is $0$, then $f(z)=1$. Otherwise, the orbit is a finite orbit of $L_\lambda$ for some $\lambda \ne 0$, and if this finite orbit is nonfixed, then $\lambda=-1$ and both points of this $2$-cycle are counted by $T_f$. Since $R_f=T_f=0$, the only nonfixed preperiodic points are the $\eta_d$ nonfixed preimages of $1$. Thus $\#\Prep_{\bQ}(f)=F_f+\eta_d$.

If $f(1)=0$, then $0$ and $1$ form a $2$-cycle, and $U_f=0$ implies that $1$ is the only rational preimage of $0$. Moreover, if $z \ne 0$ and $f(z)=1$, then $\Delta_f(z)=-z\ne0$ so $D_f=0$ implies that $z \notin \Prep_{\bQ}(f)$. Hence for every $z \in \Prep_{\bQ}(f) \setminus \{0,1\}$, the $f$-orbit of $z$ contains neither $0$ nor $1$. 
By $D_f=0$ and \eqref{eq2f}, the function $t_f$ is constant and nonzero along such an orbit. Therefore the orbit is a finite orbit of $L_\lambda$ for some $\lambda \ne 0$, and $T_f=0$ implies that every preperiodic point $z \ne 0,1$ is fixed. Thus $\#\Prep_{\bQ}(f)=F_f+2$. We conclude that if $D_f=R_f=T_f=U_f=0$, then
\begin{equation*}
\#\Prep_{\bQ}(f)=F_f+\eta_dI_1(f)+2I_0(f). \qedhere
\end{equation*}
\end{proof}

\begin{lemma} \label{lem2d}
We have
\begin{equation*}
\sum_{f\in\cP_d(X)}(R_f+T_f+U_f)\ll_d X^{d-2}\log X.
\end{equation*}
For every $f\in\cP_3$, we have $R_f=T_f=0$.
\end{lemma}

\begin{proof}
Let $f=\psi_{\mf{a}}$, write $a_d=\sigma mq^2$ as above, and write $z=u/q$ for a point counted by one of the three quantities. Such a point is preperiodic, so \eqref{eq2a} applies. For each $m$, the number of possible pairs $(q,u)$ is $O_d(X/m)$. We show that, once $\sigma,m,q,u$ and $d-3$ of the coefficients $a_{d-2},\ldots,a_0$ are fixed, the remaining two coefficients are uniquely determined in the cases of $R_f$ and $T_f$; for $U_f$, the same holds with $j=f(1)\in\{0,1\}$ fixed as well.

First we estimate $R_f$. Since $R_f=0$ whenever $f(1) \ne 1$, we assume that $f(1)=1$. For $d=3$, \eqref{eq2d} gives $R_f=0$. For $d \ge 4$, $\eta_d=1$ and $f(0)=f(1)=1$ so
\begin{equation*}
R_f=\#\{z\in\bQ\setminus\{0,1\}:f(z)=1\}.
\end{equation*}
Suppose that $z$ is counted by $R_f$. Then $f(1)=f(z)=1$ gives
\begin{equation*}
a_d+\sum_{i=1}^{d-2}a_i=0,\quad a_dz^d+\sum_{i=1}^{d-2}a_i z^i=0.
\end{equation*}
These equations uniquely determine $a_1$ and $a_2$ in terms of the other coefficients, since the determinant of their coefficient matrix is $z(z-1)\ne0$.
Next suppose that $z$ is counted by $T_f$. Dividing $f(z)=1-z$ and $f(1-z)=z$ by $z$ and $1-z$, respectively, gives
\begin{equation*}
a_dz^{d-1}+\sum_{i=2}^{d-2}a_i z^{i-1}+a_1+a_0=0,\quad
a_d(1-z)^{d-1}+\sum_{i=2}^{d-2}a_i(1-z)^{i-1}+a_1+a_0=0.
\end{equation*}
For $d=3$, subtraction gives $a_3(2z-1)=0$, which is impossible. For $d\ge4$, the determinant of the coefficient matrix of $a_1,a_2$ is $1-2z\ne0$, so these coefficients are uniquely determined by the other coefficients.

Finally, suppose that $z$ is counted by $U_f$ and let $j=f(1)\in\{0,1\}$. Then $f(1)=j$ and $f(z)=0$ give
\begin{equation*}
a_1+(1-j)a_0=-a_d-\sum_{i=2}^{d-2}a_i, \quad za_1+a_0=-a_dz^d-\sum_{i=2}^{d-2}a_i z^i.
\end{equation*}
The determinant of the coefficient matrix of $a_0,a_1$ is $(1-j)z-1\ne0$, so these coefficients are uniquely determined by the other coefficients.
In each case, the other $d-3$ coefficients have $O_d(X^{d-3})$ choices. Summing over $\sigma \in \{\pm 1\}$ and $m \le X$ gives
\begin{equation*}
\sum_{f\in\cP_d(X)}(R_f+T_f+U_f)
\ll_d X^{d-3}\sum_{m\le X}\frac Xm
\ll_d X^{d-2}\log X.\qedhere
\end{equation*}
\end{proof}

Set
\begin{equation*}
F_d(X)=\sum_{f\in\cP_d(X)}F_f,
\quad H_{d,j}(X)=\#\{f\in\cP_d(X):f(1)=j\}\quad(j=0,1).
\end{equation*}
Combining Proposition \ref{prop2b} with Lemmas \ref{lem2c} and \ref{lem2d}, we obtain
\begin{equation} \label{eq2i}
A_d(X)=F_d(X)+\eta_dH_{d,1}(X)+2H_{d,0}(X)
+O_{d,\ve}(X^{d-2+\ve}).
\end{equation}

%------------------------------------------------
\subsection{Counting the main contributions to \texorpdfstring{$A_d(X)$}{Ad(X)}}

For a nonzero rational number $r=x/y$ in lowest terms, write
\begin{equation*}
V_r=V_d(\ell_d(x,y)).
\end{equation*}
For each fixed $r=x/y$, the condition $\psi_{\mf a}(r)=r$ is equivalent to the linear equation $\ell_d(x,y)\cdot\mf a=0$. We first count its integral solutions without imposing primitivity. Set
\begin{equation*}
N_r(X)=\#\{\mf{a}\in X\cB_d\cap\bZ^d: \ell_d(x,y)\cdot\mf{a}=0,\ a_d\ne0,\ a_0>0\}.
\end{equation*}
Since we will sum $N_r(X)$ over $r\in\bQ^\times$, we need an estimate for $N_r(X)$ whose error term depends explicitly on $H(r)=\max(|x|,y)$. 

For a polynomial $P\in\bR[z]$, let $\|P\|_\infty$ denote the maximum of the absolute values of its coefficients. By Gelfond's lemma \cite[Lemma A.3]{Bug04}, we have
\begin{equation} \label{eq2g}
\|P\|_\infty\|Q\|_\infty\ll_d\|PQ\|_\infty
\end{equation}
for nonzero polynomials $P,Q\in\bR[z]$ of degrees bounded in terms of $d$.

\begin{lemma} \label{lem2e}
We have $V_r\ll_d H(r)^{-d}$. For every $X\ge1$,
\begin{equation} \label{eq2m}
N_r(X)=V_rX^{d-1}
+O_d\left(1+(X/H(r))^{d-2}\right).
\end{equation}
There is a constant $c_d>0$, depending only on $d$, such that $N_r(X)=0$ whenever $H(r)>c_d\sqrt X$.
\end{lemma}

\begin{proof}
For $\mf{a}=(a_d,a_{d-2},\ldots,a_1,a_0)\in\bZ^d$, let
\begin{equation*}
P_{\mf{a}}(z)
=a_dz^d+\sum_{i=2}^{d-2}a_iz^i+(a_1-a_0)z+a_0 \in \bZ[z].
\end{equation*}
Then $\ell_d(x,y)\cdot\mf{a}=y^d P_{\mf{a}}(x/y)$, so $\ell_d(x,y)\cdot\mf{a}=0$ if and only if $yz-x$ divides $P_{\mf{a}}(z)$. Since the coefficient of $z^{d-1}$ in $P_{\mf{a}}(z)$ is $0$, Gauss' lemma implies that
\begin{equation} \label{eq2j}
P_{\mf{a}}(z)=(yz-x)\left(kyz^{d-1}+kxz^{d-2}+\sum_{j=0}^{d-3}b_jz^j\right)=:(yz-x)Q(z)
\end{equation}
for some $k,b_0,\ldots,b_{d-3}\in\bZ$. In particular, $a_d=ky^2$ and $a_0=-xb_0$. 

Thus \eqref{eq2j} gives a $\bZ$-linear bijection between $\bZ^{d-1}$ and the lattice
\begin{equation*}
\Lambda_r=\{\mf{a}\in\bZ^d:\ell_d(x,y)\cdot\mf{a}=0\}.
\end{equation*}
Since $\ell_d(x,y)$ is primitive, $\Lambda_r$ has covolume $\|\ell_d(x,y)\|_2$. Extend the $\bZ$-linear bijection $\bZ^{d-1} \to \Lambda_r$ to an $\bR$-linear isomorphism $\bR^{d-1}\to\ell_d(x,y)^\perp$, and let $K_r\subset\bR^{d-1}$ be the preimage of $\cB_d\cap\ell_d(x,y)^\perp$ under this isomorphism. Then 
\begin{equation} \label{eq2k}
\vol_{d-1}(K_r)=V_r.
\end{equation}
Since $\mf a\in X\cB_d$, we have $\|P_{\mf a}\|_\infty\ll X$, while $\|yz-x\|_\infty=\max(|x|,y)=H(r)$. By \eqref{eq2g}, 
\begin{equation*}
\|Q(z)\|_\infty =\max(|kx|,|ky|,|b_0|,\ldots,|b_{d-3}|)=\max(|k|H(r),|b_0|,\ldots,|b_{d-3}|) \ll_d X/H(r)
\end{equation*}
so every point $(k,b_0, \ldots, b_{d-3}) \in XK_r$ satisfies
\begin{equation} \label{eq2l}
|k|\ll_d X/H(r)^2,\quad |b_j|\ll_d X/H(r)\;\;(0\le j\le d-3).
\end{equation}
In particular, \eqref{eq2k} gives $V_r\ll_d H(r)^{-d}$.

Since the set $K_r$ is compact and convex, Davenport's lemma \cite{Dav51}, together with \eqref{eq2l}, gives
\begin{align*}
\#(XK_r\cap\bZ^{d-1})
&=\vol_{d-1}(XK_r)+O_d\left(1+(X/H(r))^{d-2}\right) \\
&=V_r X^{d-1}+O_d\left(1+(X/H(r))^{d-2}\right).
\end{align*}
Since $a_d=ky^2$ and $a_0=-xb_0$, the conditions $a_d\ne0$ and $a_0>0$ exclude the coordinate sections $k=0$ and $b_0=0$, respectively. These sections contribute at most the error term above, and hence \eqref{eq2m} follows.

Finally, for every vector counted by $N_r(X)$, we have $a_d=ky^2\ne0$ so $|k|\ge1$. The first bound in \eqref{eq2l} therefore gives the final assertion.
\end{proof}

Let $\mu$ denote the Möbius function, and set
\begin{equation*}
F_d^{\mathrm{all}}(X)=\sum_{r\in\bQ^\times}N_r(X).
\end{equation*}
Every nonzero vector $\mf{a} \in \bZ^d$ has a unique representation $\mf{a}=j\mf{a}'$ for some $j\ge1$ and primitive $\mf{a}' \in \bZ^d$. Hence
\begin{equation*}
F_d^{\mathrm{all}}(X)=\sum_{j\le X}F_d(X/j),
\end{equation*}
and Möbius inversion gives
\begin{equation} \label{eq2h}
F_d(X)=\sum_{j\le X}\mu(j)F_d^{\mathrm{all}}(X/j).
\end{equation}

\begin{proposition} \label{prop2f}
The series $\sum_{r\in\bQ^\times}V_r$ is absolutely convergent. For $d\ge3$, we have
\begin{equation} \label{eq2n}
F_d^{\mathrm{all}}(X)=X^{d-1}\sum_{r\in\bQ^\times}V_r
+\begin{cases}
O(X^{3/2}),&d=3,\\
O(X^2\log X),&d=4,\\
O_d(X^{d-2}),&d\ge5.
\end{cases}
\end{equation}
\end{proposition}

\begin{proof}
There are $O(H)$ elements $r \in \bQ^\times$ with $H(r)=H$, and Lemma \ref{lem2e} gives $V_r\ll_d H(r)^{-d}$. Thus
\begin{equation*}
\sum_{H(r)>T}V_r \ll_d \sum_{H>T} H^{1-d} \ll_d T^{2-d},
\end{equation*}
so $\sum_{r\in\bQ^\times}V_r$ is absolutely convergent. 

By Lemma \ref{lem2e}, $N_r(X)=0$ if $H(r)>c_d\sqrt X$. Hence for $d\ge3$, summing \eqref{eq2m} gives
\begin{equation*}
F_d^{\mathrm{all}}(X)
=X^{d-1}\sum_{H(r)\le c_d\sqrt X}V_r +O_d\left( \sum_{H\le c_d\sqrt X}H +X^{d-2}\sum_{H\le c_d\sqrt X}H^{3-d} \right).
\end{equation*}
The error sum has the order stated in \eqref{eq2n}. Moreover,
\begin{equation*}
X^{d-1}\sum_{H(r)>c_d\sqrt X}V_r
\ll_d X^{d-1}(\sqrt X)^{2-d}
=X^{d/2},
\end{equation*}
which is absorbed in the error term. Thus the truncated sum in the main term can be replaced by the full series, which gives \eqref{eq2n}. Since every $f\in\cP_d$ satisfies $f(0)=1$, the point $0$ is not a fixed point so it need not be included in the sum over $r$.
\end{proof}

Now we combine the preceding estimates to prove Theorem \ref{thm1b} for $d \ge 3$.

\begin{proof}[Proof of Theorem \ref{thm1b} for $d \ge 3$]
By \eqref{eq2h} and \eqref{eq2n}, the main term of $F_d(X)$ is
\begin{equation*}
X^{d-1}\sum_{r\in\bQ^\times}V_r \sum_{j\le X}\frac{\mu(j)}{j^{d-1}}
= X^{d-1}\sum_{r\in\bQ^\times}V_r \left( \frac{1}{\zeta(d-1)}
+O_d(X^{2-d}) \right)
= \frac{X^{d-1}}{\zeta(d-1)}\sum_{r\in\bQ^\times}V_r + O_d(X).
\end{equation*}
For $d=4$, we use the uniform error estimate $O(Y^2\log(2Y))$ for $Y=X/j\ge1$, which is equivalent to the stated error term. Thus the contribution from the error term in \eqref{eq2n} is
\begin{equation*}
\begin{cases}
O(\sum_{j\le X}(X/j)^{3/2})=O(X^{3/2}),&d=3, \\
O(\sum_{j\le X}(X/j)^2\log(2X/j))=O(X^2\log X),&d=4, \\
O_d(\sum_{j\le X}(X/j)^{d-2})=O_d(X^{d-2}),&d\ge5.
\end{cases}
\end{equation*}

Since $\ell_d^{(1)}=\ell_d(1,1)$, the number of vectors $\mf{a}\in X\cB_d\cap\bZ^d$ satisfying $a_d\ne0$, $a_0>0$ and $\psi_{\mf{a}}(1)=1$ is $N_1(X)$. Hence Lemma \ref{lem2e} gives
\begin{equation*}
N_1(X)=V_d(\ell_d^{(1)})X^{d-1}+O_d(X^{d-2}).
\end{equation*}
Similarly, $\psi_{\mf{a}}(1)=0$ if and only if $\ell_d^{(0)}\cdot\mf{a}=0$. Since $\ell_d^{(0)}$ is a primitive vector, applying Davenport's lemma as in Lemma \ref{lem2e} gives
\begin{equation*}
\#\{\mf{a}\in X\cB_d\cap\bZ^d: \ell_d^{(0)}\cdot\mf{a}=0,\ a_d\ne0,\ a_0>0\} =V_d(\ell_d^{(0)})X^{d-1}+O_d(X^{d-2}).
\end{equation*}
In both cases, applying Möbius inversion as in \eqref{eq2h} to impose primitivity, we obtain
\begin{equation*}
H_{d,j}(X)=\frac{V_d(\ell_d^{(j)})}{\zeta(d-1)}X^{d-1}
+\begin{cases}
O(X\log X),&d=3,\\
O_d(X^{d-2}),&d\ge4,
\end{cases}\quad(j=0,1).
\end{equation*}

Every $r\in\bQ^\times$ has a unique representation $r=x/y$ with
$x\ne0$, $y\ge1$ and $\gcd(x,y)=1$, so 
\begin{equation*}
\sum_{r\in\bQ^\times}V_r=\sum_{\substack{x\ne0,\ y\ge1\\ \gcd(x,y)=1}}V_d(\ell_d(x,y)).
\end{equation*}
Combining the preceding estimates with \eqref{eq2i} gives
\begin{equation} \label{eq2o}
A_d(X)=\frac{2^{d-1}}{\zeta(d)}\gamma_d X^{d-1}
+\begin{cases}
O(X^{3/2}),&d=3,\\
O_{d,\ve}(X^{d-2+\ve}),&d\ge4,
\end{cases}
\end{equation}
with $\gamma_d$ as in \eqref{eq1c}. Since $\mf{a}=(0,\ldots,0,1/2)$ is an interior point of $\cB_d$ that satisfies $\ell_d^{(1)}\cdot\mf{a}=0$, we have $V_d(\ell_d^{(1)})>0$, and hence $\gamma_d>0$. Now \eqref{eq1b} and \eqref{eq2o} complete the proof.
\end{proof}

%------------------------------------------------
%------------------------------------------------
\section{Proof of Theorem \ref{thm1b}: the case \texorpdfstring{$d=2$}{d=2}} \label{Sec3}

Let $\mf{a}=(a_2, a_0) \in \bZ_{\prim}^2$ with $a_2 \ne 0$ and $a_0>0$, and set $c=\frac{a_2}{a_0}\in\bQ^\times$ and $L_c(z)=cz$. Then $\psi_{\mf{a}}(z)=cz^2+1$ and 
\begin{equation*}
(L_c\circ \psi_{\mf{a}} \circ L_c^{-1})(z)=z^2+c=:f_c(z),
\end{equation*}
so $L_c$ gives a bijection between the sets $\Prep_{\bQ}(\psi_{\mf{a}})$ and $\Prep_{\bQ}(f_c)$. In the rest of this section, we work with $f_c$ in place of $\psi_{\mf{a}}$.
In particular,
\begin{equation} \label{eq3a}
A_2(X)=\#\{(c,z)\in\bQ^\times\times\bQ:
H(c)\le X,\ z\in\Prep_{\bQ}(f_c)\}.
\end{equation}

%------------------------------------------------
\subsection{A surface from seven consecutive orbit points}

First we provide two lemmas which follow from of the local bounds of Benedetto, Chen, Hyde, Kovacheva and
White \cite[Lemma 3.1]{BCHKW14} for rational preperiodic points of $z^2+c$.

\begin{lemma} \label{lem3a}
Let $c\in\bQ^\times$ and suppose that $\Prep_{\bQ}(f_c)\ne\varnothing$. Then there are unique integers $a\ne0$ and $q\ge1$ such that $\gcd(a,q)=1$ and $c=a/q^2$. Moreover, every point in the forward orbit of any $z\in\Prep_{\bQ}(f_c)$ has reduced denominator $q$.
\end{lemma}

\begin{proof}
By \cite[Lemma 3.1]{BCHKW14}, if a rational preperiodic point has reduced denominator $q$, then $c$ has reduced denominator $q^2$. Applying this to each point in the forward orbit of any $z\in\Prep_{\bQ}(f_c)$ shows that every point in that orbit has reduced denominator $q$.
\end{proof}

\begin{lemma} \label{lem3b}
Let $c\in\bQ$ and $z\in\Prep_{\bQ}(f_c)$. Then $|z|\le2\max(1,|c|^{1/2})$.
\end{lemma}

\begin{proof}
By \cite[Lemma 3.1]{BCHKW14}, we have $|z|\le2$ when $c\ge-2$.
When $c<-2$, the same lemma gives
\begin{equation*}
|z|\le\frac{1+\sqrt{1-4c}}{2}<2|c|^{1/2}. \qedhere
\end{equation*}
\end{proof}

Write $u=z$ and $v=f_c(z)$. Since $c=v-u^2$, these two orbit points determine both the parameter $c$ and every subsequent iterate $f_c^i(z)$ for $i\ge2$. Define polynomials $G_0, G_1,\ldots, G_6 \in\bZ[u,v]$ by
\begin{equation} \label{eq3b}
G_0(u,v)=u,\quad
G_{i+1}(u,v)=G_i(u,v)^2-u^2+v\quad(0\le i\le5).
\end{equation}
Then $G_1(u,v)=v$ and $G_i(z,f_c(z))=f_c^i(z)$ for $0\le i\le6$.
Let $\mc{S}\subset\mf{P}^3_{\bQ}$ be the projective hypersurface in homogeneous coordinates $[q:x_0:x_1:x_6]$ defined by
\begin{equation} \label{eq3c}
q^{31}x_6=q^{32}G_6(x_0/q,x_1/q).
\end{equation}

We use the sixth iterate because Lemma \ref{lem3d} requires an $O(B)$ bound for the number of rational points on $\mc S$ of height at most $B$ that lie on none of its lines. For a geometrically integral surface of degree $D$ in $\mf{P}^3$, Salberger's theorem \cite[Theorem 0.9(b)]{Sal23} gives $O_D(B^{3/(2\sqrt D)}\log B)$ geometrically integral curves of bounded degree that contain all but $O_{D,\ve}(B^{3/\sqrt D+\ve})$ of the rational points of height at most $B$. By \cite[Theorem 1.17]{Sal23}, each of these curves of degree at least $3$ contains $O_D(B^{2/3}\log B)$ such points. Summing over these curves gives a total contribution $O_D(B^{3/(2\sqrt D)+2/3}(\log B)^2)$.

The projective closure of the graph $w=G_5(u,v)$ has degree $16$, so the exponent in this last bound is $3/8+2/3=25/24>1$, which does not give the required $O(B)$ bound. The projective closure of the graph $w=G_6(u,v)$ is the surface $\mc S$ of degree $32$. Since $3/(2\sqrt{32})+2/3<1$, the curves of degree at least $3$ provided by Salberger's theorem together contain $o(B)$ rational points of height at most $B$. Taking $\ve>0$ sufficiently small gives $3/\sqrt{32}+\ve<1$, so the points outside the covering curves also contribute $o(B)$. Therefore, it is enough to treat the lines and conics on $\mc S$ separately.

\begin{lemma} \label{lem3c}
The surface $\mc{S}$ is geometrically integral of degree $32$ and contains exactly six lines over $\overline{\bQ}$. Two lines lie in $q=0$, and the remaining four lines are the closures of the affine lines given on $q=1$ by
\begin{equation} \label{eq3d}
(x_1,x_6)=(x_0,x_0),\,
(-x_0,-x_0),\,
(-x_0-1,x_0),\,
(x_0-1,-x_0).
\end{equation}
\end{lemma}

\begin{proof}
By \eqref{eq3b}, $\deg G_6=32$ and its homogeneous part of degree $32$ is $(v^2-u^2)^{16}$. The defining polynomial in \eqref{eq3c} is the homogenization of the absolutely irreducible polynomial $x_6-G_6(x_0,x_1)$. Thus $\mc{S}$ is geometrically integral of degree $32$, and the reduced intersection $(\mc{S}\cap\{q=0\})_{\text{red}}$ is the union of the two lines $x_1=x_0$ and $x_1=-x_0$ in the plane $q=0$.

Let $L\subset\mc{S}_{\overline{\bQ}}$ be a line meeting the affine chart $q\ne0$. On this chart, normalize $q=1$ and choose an affine parameter $s$ on $L\cap\{q\ne0\}$. Then $x_0(s)$, $x_1(s)$ and $x_6(s)$ are affine in $s$. We claim that $G_i(x_0(s),x_1(s))$ is affine in $s$ for every $0\le i\le 6$. It is clear that $G_6(x_0(s),x_1(s))=x_6(s)$ is affine in $s$. If
\begin{equation*}
\deg_s G_i(x_0(s),x_1(s))\ge2
\end{equation*}
for some $0\le i\le5$, then \eqref{eq3b} gives
\begin{equation*}
\deg_s G_{i+1}(x_0(s),x_1(s))=2\deg_s G_i(x_0(s),x_1(s)).
\end{equation*}
Iterating this, we have $\deg_s G_6(x_0(s),x_1(s))>1$, which is a contradiction.

If both $x_0(s)$ and $x_1(s)$ are constant, then $x_6(s)=G_6(x_0(s),x_1(s))$ is also constant, contradicting the fact that $L$ is a line. If exactly one of $x_0(s)$ and $x_1(s)$ is constant, then 
\begin{equation*}
G_2(x_0(s),x_1(s))=x_1(s)^2-x_0(s)^2+x_1(s)
\end{equation*}
has degree $2$ in $s$, which is impossible. Hence both $x_0(s)$ and $x_1(s)$ are nonconstant. Write $x_1=\lambda x_0+\mu$ for some $\lambda,\mu\in\overline{\bQ}$. Then
\begin{equation*}
G_2(x_0,x_1)=(\lambda^2-1)x_0^2+\lambda(2\mu+1)x_0+\mu(\mu+1)
\end{equation*}
is affine, so $\lambda^2=1$. Moreover, 
\begin{equation*}
G_3(x_0,x_1)=(\lambda^2(2\mu+1)^2-1)x_0^2+2\lambda(2\mu+1)\mu(\mu+1)x_0+x_1+\mu^2(\mu+1)^2
\end{equation*}
is affine, so $(2\mu+1)^2=1$. Hence $\lambda\in\{1,-1\}$ and $\mu\in\{0,-1\}$, so $x_1$ is one of $x_0$, $-x_0$, $-x_0-1$ and $x_0-1$. The corresponding values of $x_6=G_6(x_0,x_1)$ are $x_0$, $-x_0$, $x_0$ and $-x_0$, respectively. This proves \eqref{eq3d}.
\end{proof}

Under the substitution $z=x_0$ and $c=x_1-x_0^2$, the four affine lines in \eqref{eq3d} correspond to
\begin{equation} \label{eq3e}
c=z-z^2,\quad c=-z-z^2,\quad
c=-z-z^2-1,\quad c=z-z^2-1.
\end{equation}
The first two relations say that $z$ and $-z$, respectively, are fixed by $f_c$. Set $w=z$ for the third relation, and $w=-z$ for the fourth relation. In both cases,
$f_c(w)=-w-1$ and $f_c(-w-1)=w$. Thus $w$ has period dividing two. Since $f_c(z)=f_c(-z)$, all four relations give a preperiodic point $z$.

For a locally closed subset $V\subset\mf{P}^n_{\bQ}$, write
\begin{equation*}
N(V;B)=\#\{P\in V(\bQ):H(P)\le B\},
\end{equation*}
where $H([y_0:\cdots:y_n])=\max_i|y_i|$ for primitive integral coordinates $(y_0,\ldots,y_n)$.

The proof of the next lemma is based on the work of Salberger \cite{Sal23}. By a \emph{geometric line} on $\mc{S}$, we mean a line contained in $\mc{S}_{\overline{\bQ}}$.

\begin{lemma} \label{lem3d}
Let $\mc{S}^\circ$ be the complement of all geometric lines on $\mc{S}$. Then $N(\mc{S}^\circ;B)\ll B$.
\end{lemma}

\begin{proof}
Let $\mc{C}$ be the set of geometrically integral conics on $\mc{S}$ defined over $\bQ$. We first show that $\mc{C}$ consists of the two conics
\begin{equation} \label{eq3g}
\begin{aligned}
C_1:&\quad x_0^2=x_1^2+2qx_1,\quad x_6=-x_1,\\
C_2:&\quad x_0^2=x_1^2+q^2,\quad x_6=x_1-q.
\end{aligned}
\end{equation}
Let $C\subset\mc{S}_{\overline{\bQ}}$ be an integral conic. By the proof of Lemma \ref{lem3c}, the reduced intersection of $\mc{S}$ with $q=0$ is a union of two lines, so $C$ is not contained in $q=0$. Parametrize $C$ by binary quadratic forms over $\overline{\bQ}$ as
\begin{equation*}
[q:x_0:x_1:x_6]=[Q(s,t):X_0(s,t):X_1(s,t):X_6(s,t)]
\end{equation*}
and set $X_i=QG_i(X_0/Q,X_1/Q)$ for $2\le i\le5$. By \eqref{eq3b}, these rational functions satisfy
\begin{equation*}
X_i^2=X_0^2-QX_1+QX_{i+1}\quad(1\le i\le5).
\end{equation*}
Applying this identity successively for $i=5,\ldots,2$, we see that every $X_i$ is a quadratic form.

We claim that $Q,X_0,X_1$ are linearly independent. Otherwise $C$ would be the projective closure of the graph of $G_6$ over an affine line $\ell$ in the $(u,v)$-plane. The calculation in the proof of Lemma \ref{lem3c} shows that the restriction of $G_6$ to $\ell$ has degree $1$, or the restriction of one of $G_2,G_3$ to $\ell$ has degree $2$. In the latter case, \eqref{eq3b} shows that the restriction of $G_6$ to $\ell$ has degree $32$ or $16$, respectively. Both cases contradict the fact that $C$ is a conic. Hence $Q,X_0,X_1$ are linearly independent so they form a basis of the $\overline{\bQ}$-vector space of binary quadratic forms in $s,t$. 
For each $i$, write
\begin{equation*}
X_i=\alpha_i Q+\beta_i X_0+\gamma_i X_1
\end{equation*}
for some $\alpha_i, \beta_i, \gamma_i \in \overline{\bQ}$. On the affine chart $q\ne0$, set $u=X_0/Q$ and $v=X_1/Q$. Dividing by $Q$, we see that every $G_i(u,v)=X_i/Q$ is affine in $u,v$ on $C$.

Write $G_2(u,v)=au+bv+r$ on $C$, where $a,b,r\in\overline{\bQ}$. Since $G_2(u,v)=v^2-u^2+v$, the projection of $C\cap\{q \ne 0\}$ to the $(u,v)$-plane is an integral conic with equation
\begin{equation*}
v^2-u^2-au+(1-b)v-r=0.
\end{equation*}
Moreover, $G_3(u,v)=(au+bv+r)^2-u^2+v$ is affine on $C$, so its quadratic part is a scalar multiple of $v^2-u^2$, the quadratic part of the equation of $C$. This gives $ab=0$ and $a^2+b^2=1$. Thus $G_2(u,v)$ is either $\sigma u+r$ or $\sigma v+r$, where $\sigma\in\{1,-1\}$. 

If $G_2(u,v)=\sigma u+r$, then $G_3(u,v)=2\sigma r\,u+v+r^2$. Since $G_4(u,v)$ is affine on $C$, comparison of quadratic terms gives $r=0$, and the equation of $C$ becomes $(v-\sigma u)(v+\sigma u+1)=0$. Thus $C$ is reducible, which is a contradiction. If $G_2(u,v)=\sigma v+r$, then
\begin{equation*}
u^2=v^2+(1-\sigma)v-r,\quad
G_3(u,v)=\sigma(1+2r)v+r(r+1).
\end{equation*}
The same comparison gives $(1+2r)^2=1$, so $r\in\{0,-1\}$. If $(\sigma,r)=(1,0)$ or $(-1,-1)$, then $C$ is reducible, which is a contradiction. The remaining cases $(\sigma,r)=(-1,0)$ and $(1,-1)$ give $u^2=v^2+2v$ and $u^2=v^2+1$, respectively. In these cases, \eqref{eq3b} gives $G_6(u,v)=-v$ and $G_6(u,v)=v-1$, respectively, so $C$ is one of the conics in \eqref{eq3g}. Conversely, both conics in \eqref{eq3g} are geometrically integral and satisfy \eqref{eq3c}. This proves that $\mc{C}=\{C_1,C_2\}$.

Set $\alpha=3/(2\sqrt{32})$. By \cite[Theorem 0.9(b)]{Sal23}, there is a finite set $\{D_\lambda:\lambda\in\Lambda\}$ of geometrically integral curves of bounded degree on $\mc{S}$ such that $\#\Lambda\ll B^\alpha\log B$ and all but $O(B^{2\alpha+1/4})$ rational points of $\mc{S}$ of height at most $B$ lie on $\bigcup_{\lambda \in \Lambda} D_\lambda$. By \cite[Theorem 1.17]{Sal23}, we have
\begin{equation*}
\sum_{\substack{\lambda\in\Lambda\\ \deg D_\lambda\ge3}}N(D_\lambda;B)
\ll B^{\alpha+2/3}(\log B)^2.
\end{equation*}
Every $D_\lambda$ of degree $2$ is an element of $\mc{C}$, so \cite[Theorem 1.17]{Sal23} gives
\begin{equation*}
\sum_{\substack{\lambda\in\Lambda\\ \deg D_\lambda=2}}N(D_\lambda;B)
\le N(C_1;B)+N(C_2;B)\ll B.
\end{equation*}
The curves $D_\lambda$ of degree $1$ are lines, which do not contribute to $N(\mc{S}^\circ;B)$. We conclude that
\begin{equation*}
N(\mc{S}^\circ;B)\ll B+B^{2\alpha+1/4}+B^{\alpha+2/3}(\log B)^2\ll B. \qedhere
\end{equation*}
\end{proof}

\begin{lemma} \label{lem3e}
The number of pairs in \eqref{eq3a} that satisfy none of the four possibilities in \eqref{eq3e} is $O(\sqrt X)$.
\end{lemma}

\begin{proof}
Consider a pair $(c,z)$ counted in \eqref{eq3a}. By Lemma \ref{lem3a}, we have $c=a/q^2$, where $a\ne0$, $q\ge1$, $\gcd(a,q)=1$ and $|a|,q^2\le X$. By Lemmas \ref{lem3a} and \ref{lem3b}, we may write
\begin{equation*}
f_c^i(z)=\frac{x_i}{q},\quad |x_i|\le2\max(q,|a|^{1/2})\le2\sqrt X \;\;(i=0,1,6)
\end{equation*}
for some integers $x_0, x_1, x_6$. Since $G_6(z,f_c(z))=f_c^6(z)$, the point $[q:x_0:x_1:x_6]$ lies on $\mc{S}$. Its coordinates are primitive because $\gcd(q,x_0)=1$, so its height is at most $2\sqrt X$. The map $(c,z) \mapsto [q:x_0:x_1:x_6]$ is injective, since
\begin{equation*}
z=\frac{x_0}{q},\quad
c=\frac{qx_1-x_0^2}{q^2}.
\end{equation*}
Let $\mc{S}^\circ$ be the complement of all geometric lines on $\mc{S}$. By Lemma \ref{lem3c}, the pairs $(c,z)$ satisfying none of the four possibilities in \eqref{eq3e} map to $\mc{S}^\circ$. Lemma \ref{lem3d} bounds their number by $N(\mc{S}^\circ;2\sqrt X)\ll\sqrt X$.
\end{proof}

%------------------------------------------------
\subsection{Counting points on the four lines}

Define
\begin{align*}
N_F(X)&=\#\{z\in\bQ:H(z-z^2)\le X\}=\#\{z\in\bQ:H(-z-z^2)\le X\},\\
N_C(X)&=\#\{z\in\bQ:H(-z-z^2-1)\le X\}=\#\{z\in\bQ:H(z-z^2-1)\le X\}.
\end{align*}
The equalities follow from the bijection $z\mapsto -z$ on $\bQ$. By Lemma \ref{lem3e}, the pairs counted by \eqref{eq3a} that satisfy none of the four possibilities in \eqref{eq3e} contribute $O(\sqrt{X})$. The four possibilities in \eqref{eq3e} contribute $N_F(X),N_F(X),N_C(X),N_C(X)$, respectively, up to $O(1)$ points coming from intersections of distinct lines and from the excluded case $c=0$. Hence
\begin{equation} \label{eq3h}
A_2(X)=2N_F(X)+2N_C(X)+O(\sqrt X).
\end{equation}

\begin{lemma} \label{lem3f}
We have
\begin{equation} \label{eq3i}
\begin{split}
N_F(X)&=\frac{X}{\zeta(2)}\left(\frac{\sqrt5}{2}+2\log \frac{1+\sqrt{5}}{2} \right)+O(\sqrt X\log X),\\
N_C(X)&=\frac{X}{\zeta(2)}\left(\frac12+\frac{2\pi}{3\sqrt3}\right)+O(\sqrt X\log X).
\end{split}
\end{equation}
\end{lemma}

\begin{proof}
For $j\in\{1,-3\}$, let $\phi_j(t)=j/4-t^2$. Then
$$
z-z^2=\phi_1(z-1/2),\quad -z-z^2-1=\phi_{-3}(z+1/2).
$$
Each value in $\phi_j(\bQ^\times)$ has exactly two rational preimages. The remaining input $t=0$ gives the value $j/4$, which has height $4$. Hence, for $X\ge4$,
$$
N_F(X)=2V_1(X)+1,\quad N_C(X)=2V_{-3}(X)+1,
$$
where $V_j(X)=\#\{c\in\phi_j(\bQ^\times):H(c)\le X\}$.
Olechnowicz \cite[proof of Theorem 8.1, Step 4]{Ole26} obtained
$$
V_j(X)=\frac{3\gamma(j)}{\pi^2}X+O(\sqrt X\log X),
$$
where
$$
\gamma(1)=\frac{\sqrt5}{2}+2\log \frac{1+\sqrt{5}}{2},\quad  \gamma(-3)=\frac12+\frac{2\pi}{3\sqrt3}.
$$
Since $6/\pi^2=1/\zeta(2)$, these estimates give \eqref{eq3i}.
\end{proof}

\begin{proof}[Proof of Theorem \ref{thm1b} for $d=2$]
Substituting \eqref{eq3i} into \eqref{eq3h}, we obtain
\begin{equation*}
A_2(X)=\frac{2}{\zeta(2)}\left(\frac{1+\sqrt{5}}{2}+2\log \frac{1+\sqrt{5}}{2} +\frac{2\pi}{3\sqrt3}\right)X+O(\sqrt X\log X).
\end{equation*}
Dividing by $\#\cP_2(X)$ and using \eqref{eq1b} with $d=2$ completes the proof.
\end{proof}

%------------------------------------------------
\medskip
\section*{Acknowledgments}
The author was supported by the National Research Foundation of Korea (NRF) grant funded by the Korea government (MSIT) (No. RS-2024-00334558 and No. RS-2025-02262988).

%------------------------------------------------
\medskip
\section*{Statement on AI use}

OpenAI's ChatGPT 5.6 Sol and 6 Astra generated the initial proof of Theorem \ref{thm1b} through an iterative dialogue with the author. 
In an earlier version of the paper, ChatGPT 5.6 Sol produced the weaker bounds $M_d(X)\ll_d X^{-1}\exp(C\sqrt{\log X/\log\log X})$ for every real number $C>4\sqrt{2}$ when $d\ge3$, and $M_2(X)\ll X^{-1}\log X$. The author then extended the argument based on three consecutive orbit points to four consecutive orbit points, which removed the factor $\log X$ and gave $M_2(X)\ll X^{-1}$. Subsequently, ChatGPT 6 Astra generated the proof of the stronger asymptotic formula stated in Theorem \ref{thm1b}. 

The author formulated the project, evaluated and revised the models' suggestions, independently verified all mathematical arguments, and wrote the final manuscript. The author takes full responsibility for the correctness of the results.

%------------------------------------------------

\end{document}